\documentclass[english,oneside,a4paper,11pt]{article}
\usepackage{amssymb}
\usepackage{latexsym}
\usepackage[english]{babel}
\usepackage{amsmath}
\usepackage{amsfonts}
\usepackage{amsbsy}
\usepackage{mathrsfs}
\usepackage[T1]{fontenc}
\usepackage[utf8]{inputenc}
\usepackage{verbatim}
\usepackage{mathrsfs}
\usepackage{graphicx}
\usepackage{float}
\usepackage{bbm}
\usepackage{lscape}
\usepackage{empheq}
\usepackage[dvipsnames]{xcolor}
\usepackage{array}
\usepackage{a4wide}
\usepackage{enumitem}
\usepackage[dvips,colorlinks,bookmarks,breaklinks]{hyperref}  
\hypersetup{linkcolor=black,citecolor=black,filecolor=black,urlcolor=black}

\usepackage{xcolor}
\usepackage{soul}

\providecommand{\abs}[1]{\left\lvert#1 \right\rvert}

\newtheorem{theorem}{Theorem}

\newtheorem{corollary}{Corollary}

\newenvironment{proof}
{\begin{trivlist}\item[\hskip%
\labelsep{{\it \noindent Proof.}}]}{\hfill $\square$
\end{trivlist}}

\newcounter{counter}
\newcommand{\counter}{\stepcounter{counter}\thecounter}

\newenvironment{remark}
{\begin{trivlist}\item[\hskip%
\labelsep{{\it \noindent Remark \counter}}]}{\hfill
\end{trivlist}}

\numberwithin{equation}{section}

\allowdisplaybreaks

\begin{document}
\begin{center}
{\huge {\bf A maximal inequality for \\
dependent random variables}} \\
\vspace{1.0cm}
{\Large Jo\~{a}o Lita da Silva\footnote{\textit{E-mail address:} \texttt{jfls@fct.unl.pt}; \texttt{joao.lita@gmail.com} (corresponding author)}} \\
\vspace{0.1cm}
\textit{Department of Mathematics and GeoBioTec \\ NOVA School of Sciences and Technology \\
NOVA University of Lisbon \\ Quinta da Torre, 2829-516 Caparica,
Portugal}
\end{center}

\vspace{1.5cm}

\begin{abstract}
For a sequence $\{X_{n}, \, n \geqslant 1 \}$ of random variables satisfying $\mathbb{E} \lvert X_{n} \rvert < \infty$ for all $n \geqslant 1$, a maximal inequality is established, and used to obtain strong law of large numbers for dependent random variables.
\end{abstract}

\bigskip


\bigskip


\bigskip

\section{Introduction}

\indent

The most notorious maximal inequality in Probability Theory is, perhaps, the Kolmogorov inequality \cite{Kolmogoroff28} published in 1928: if $\left\{X_{n}, \, n \geqslant 1 \right\}$ is a sequence of (mutually) independent random variables with mean values
$\mathbb{E} \, X_{n} = 0$ and 
variances $\mathbb{V}(X_{n}) < \infty$ for each $n \geqslant 1$ then, for any $\varepsilon > 0$,
\begin{equation*}
\mathbb{P} \left\{\max_{1 \leqslant k \leqslant n} \lvert X_{1} + X_{2} + \ldots + X_{k} \rvert > \varepsilon \right\} \leqslant \frac{1}{\varepsilon^{2}} \sum_{k=1}^{n} \mathbb{V}(X_{k}).
\end{equation*}
Since then, several authors have extended Kolmogorov's inequality in many ways, namely by weakening the (mutual) independence assumption of the random variables (see \cite{Bhattacharya05}, \cite{Chow60}, \cite{Christofides00}, \cite{Eghbal10}, \cite{Hajek55}, \cite{Kounias69}, \cite{Matula92}, \cite{Newman82}, \cite{Rao02}, \cite{Rio95b}, \cite{Shao95} among others).

It is commonly known that the significant role of maximal inequalities stems from the fact that these are a cornerstone in the obtention of sharp limit theorems.
Recently, it was proved in \cite{Thanh20} the Baum-Katz theorem \cite{Baum65} for sequences of pairwise independent and identically distributed random variables under optimal moment conditions, by using a technique developed in \cite{Rio95a}. The key ingredient for this achievement is an asymptotic maximal inequality which takes a particular format on the scenario of pairwise independence and identical distribution admitted for the random sequence. In this paper, the aforementioned asymptotic maximal inequality is brought to light in a general setting, i.e. the random variables $X_{1}, X_{2}, \ldots$ are not assumed to be independent, nor having any particular dependence structure. As consequence, Strong Law of Large Numbers (SLLN) with respect to optimal norming constants can be obtained for sequences of dependent random variables. Over the last decades, many stochastic models involving dependent random variables have emerged in mathematical statistics, statistical physics or reliability theory (see \cite{Bulinski07} and \cite{Hutchinson90}), whence the importance of having statements (particularly, SLLN) that allow us to deal with it.

We shall need to introduce some relevant notations which will be employed along this paper. Associated to a probability space $(\Omega, \mathcal{F}, \mathbb{P})$, we shall consider the space $\mathscr{L}_{p}$ $(p > 0)$ of all measurable functions $X$ (necessarily random variables) for which $\mathbb{E} \abs{X}^{p} < \infty$. We shall write $x \wedge y$ and $x \vee y$ for $\min\{x,y\}$ and $\max\{x,y\}$, respectively. For any measurable function $X$, we shall define its positive and negative parts by $X^{+} = 0 \vee X$ and $X^{-} = 0 \vee (-X)$, respectively. Given an event $A$, we shall denote the indicator random variable of the event $A$ by $I_{A}$. Throughout, $\lfloor x \rfloor$ shall be used to represent the largest integer not greater than $x$. To make the computations be simpler looking, we shall employ the letter $C$ to denote any positive constant that can be explicitly computed and whose value may change between appearances; the symbol $C(r)$ has identical meaning with the additional information that the constant depends on $r$.

\section{Maximal inequality}

\indent

Let $L,K$ be real numbers such that $0 < L \leqslant K$. We shall define the function $g_{L}(t) = (t \wedge L) \vee (-L)$, which describes the truncation at level $L$, and
\begin{equation*}
h_{L,K}(t) = \lvert g_{K}(t) - g_{L}(t) \rvert = 0 \vee [(K - L) \wedge (t - L)] - 0 \wedge [(L - K) \vee (t + L)].
\end{equation*}
Considering two random variables $X,Y$ we shall put
\begin{equation*}
\Delta_{X,Y}(x,y) := \mathbb{P} \left\{X \leqslant x, Y \leqslant y \right\} - \mathbb{P} \left\{X \leqslant x \right\} \mathbb{P} \left\{Y \leqslant y \right\}
\end{equation*}
and define the following covariance quantities
\begin{align}
G_{X,Y}(L) &:= \mathrm{Cov} \big(g_{L}(X),g_{L}(Y) \big) = \int_{-L}^{L} \int_{-L}^{L} \Delta_{X,Y}(u,v) \, \mathrm{d}u \, \mathrm{d}v, \label{eq:2.1} \\[10pt]
\begin{split}\label{eq:2.2}
H_{X,Y}(L,K) &:= \mathrm{Cov} \big(h_{L,K}(X),h_{L,K}(Y) \big) \\
&= \int_{L}^{K} \int_{L}^{K} \Delta_{X,Y}(u,v) \, \mathrm{d}u \, \mathrm{d}v + \int_{-K}^{-L} \int_{-K}^{-L} \Delta_{X,Y}(u,v) \, \mathrm{d}u \, \mathrm{d}v \\
&\quad - \int_{L}^{K} \int_{-K}^{-L} \Delta_{X,Y}(u,v) \, \mathrm{d}u \, \mathrm{d}v - \int_{-K}^{-L} \int_{L}^{K} \Delta_{X,Y}(u,v) \, \mathrm{d}u \, \mathrm{d}v
\end{split}
\end{align}
where the last identities in \eqref{eq:2.1} and \eqref{eq:2.2} are consequence of Theorem 2.3 in \cite{Yu93}. It should be noted that if $X,Y$ are independent random variables then $\Delta_{X,Y}(u,v) = 0$ for any reals $u,v$.

The main result is given in the next:

\begin{theorem}\label{thr:2.1}
Let $r > 1$ be an integer and $\left\{X_{n}, \, n \geqslant 1 \right\}$ a sequence of random variables verifying $\mathbb{E} \lvert X_{n} \rvert < \infty$ for all $n$. If $\{b_{n} \}$ is nondecreasing sequence of positive constants, and $\{a_{n,k}, \, 1 \leqslant k \leqslant n + 1, n \geqslant 1 \}$ is an array of positive constants satisfying \\

\noindent \textnormal{(a)} ${\displaystyle \sum_{k=1}^{n+1} a_{n,k} = O(b_{r^{n}})}$ as $n \rightarrow \infty$, \\

\noindent \textnormal{(b)} ${\displaystyle \sum_{k=1}^{n+1} \max_{0 \leqslant h < r^{n-k+1}} \sum_{j=1 + h r^{k}}^{h r^{k} + r^{k}} \mathbb{E} \lvert X_{j} \rvert I_{\left\{\lvert X_{j} \rvert > b_{r^{k-1}} \right\}} = o(b_{r^{n}})}$ as $n \rightarrow \infty$, \\

\noindent then, for any $\varepsilon > 0$, there is a positive integer $n_{0}$ and a constant $C(r) > 0$ such that
\begin{equation*}
\begin{split}
    &\mathbb{P} \left\{\max_{1 \leqslant m < r^{n+1}} \Bigg\lvert\sum_{k=1}^{m} (X_{k} - \mathbb{E} \, X_{k}) \Bigg\rvert > \varepsilon b_{r^{n}} \right\} \\
    &\quad \leqslant \sum_{k=1}^{r^{n+1} - 1} \mathbb{P} \left\{\lvert X_{k} \rvert > b_{r^{n+1}} \right\} + \frac{C(r)}{\varepsilon^{2}} \sum_{k=1}^{n+1} \sum_{j=1}^{r^{n+1}} a_{n,k}^{-2} \big(\mathbb{E} X_{j}^{2} I_{\left\{\lvert X_{j} \rvert \leqslant b_{r^{k}} \right\}} + b_{r^{k}}^{2} \mathbb{P} \big\{\lvert X_{j} \rvert > b_{r^{k-1}} \big\} \big) \\
    &\qquad + \frac{C(r)}{\varepsilon^{2}} \sum_{k=1}^{n+1} a_{n,k}^{-2} \max_{1 \leqslant \ell \leqslant r-1} \sum_{h=0}^{r^{n-k+1} - 1} \left[\sum_{1 + hr^{k} \leqslant i < j \leqslant h r^{k} + \ell r^{k-1}}G_{X_{i},X_{j}}(b_{r^{k-1}}) \right]^{+} \\
    &\qquad + \frac{C(r)}{\varepsilon^{2}} \sum_{k=1}^{n+1} a_{n,k}^{-2} \sum_{h=0}^{r^{n-k+1} - 1} \left[\sum_{1 + hr^{k} \leqslant i < j \leqslant h r^{k} + r^{k}} H_{X_{i},X_{j}}(b_{r^{k-1}},b_{r^{k}}) \right]^{+}
\end{split}
\end{equation*}
for all $n \geqslant n_{0}$.
\end{theorem}

\begin{proof}
Let $r > 1$ be a (fixed) integer. Setting
\begin{gather*}
X_{k,b_{r^{n+1}}}' := g_{b_{r^{n+1}}}(X_{k}) = X_{k} I_{\left\{\lvert X_{k} \rvert \leqslant b_{r^{n+1}} \right\}} + b_{r^{n+1}} I_{\left\{X_{k} > b_{r^{n+1}} \right\}} - b_{r^{n+1}} I_{\left\{X_{k} < - b_{r^{n+1}} \right\}}, \\
\Gamma_{n} := \bigcap_{k=1}^{r^{n+1} - 1} \{X_{k,b_{r^{n+1}}}' = X_{k} \}, \\
S_{m,b_{r^{n+1}}} := \sum_{k=1}^{m} (X_{k,b_{r^{n+1}}}' - \mathbb{E} X_{k,b_{r^{n+1}}}'), \qquad \quad n,m \geqslant 0
\end{gather*}
we obtain
\begin{align*}
&\mathbb{P} \left\{\max_{1 \leqslant m < r^{n+1}} \Bigg\lvert \sum_{j=1}^{m} (X_{j} - \mathbb{E} X_{j}) \Bigg\rvert > \varepsilon b_{r^{n}} \right\} \\
&\quad = \mathbb{P} \left[\left\{\max_{1 \leqslant m < r^{n+1}} \Bigg\lvert \sum_{j=1}^{m} (X_{j} - \mathbb{E} X_{j}) \Bigg\rvert > \varepsilon b_{r^{n}} \right\} \cap \Gamma_{n} \right] \\
&\qquad + \mathbb{P} \left[\left\{\max_{1 \leqslant m < r^{n+1}} \Bigg\lvert \sum_{j=1}^{m} (X_{j} - \mathbb{E} X_{j}) \Bigg\rvert > \varepsilon b_{r^{n}} \right\} \cap \Gamma_{n}^{\complement} \right] \\
&\quad \leqslant \mathbb{P} \left\{\max_{1 \leqslant m < r^{n+1}} \lvert S_{m,b_{r^{n+1}}} \rvert > \varepsilon b_{r^{n}} \right\} + \sum_{k=1}^{r^{n+1} - 1} \mathbb{P} \left\{\lvert X_{k} \rvert > b_{r^{n+1}} \right\}.
\end{align*}
Moreover, for each $1 \leqslant m < r^{n+1}$, we have
\begin{equation}\label{eq:2.3}
\begin{split}
S_{m,b_{r^{n+1}}} &= S_{m,b_{1}} + \sum_{k=1}^{n+1} (S_{m,b_{r^{k}}} - S_{m,b_{r^{k-1}}}) \\
&= (\overset{= S_{m,b_{1}}}{\overbrace{S_{r^{0} \lfloor \frac{m}{r^{0}} \rfloor,b_{r^{0}}}}} - \underset{= S_{0,b_{r^{n+1}}} = 0}{\underbrace{S_{r^{n+1} \lfloor \frac{m}{r^{n+1}} \rfloor,b_{r^{n+1}}}}}) + \sum_{k=1}^{n+1} (S_{m,b_{r^{k}}} - S_{m,b_{r^{k-1}}}) \\
&= \sum_{k=1}^{n+1} (S_{r^{k-1} \lfloor \frac{m}{r^{k-1}} \rfloor,b_{r^{k-1}}} - S_{r^{k} \lfloor \frac{m}{r^{k}} \rfloor,b_{r^{k}}}) + \sum_{k=1}^{n+1} (S_{m,b_{r^{k}}} - S_{m,b_{r^{k-1}}}) \\
&= \sum_{k=1}^{n+1} (S_{r^{k-1} \lfloor \frac{m}{r^{k-1}} \rfloor,b_{r^{k-1}}} - S_{r^{k} \lfloor \frac{m}{r^{k}} \rfloor,b_{r^{k-1}}}) \\
&\qquad + \sum_{k=1}^{n+1} (S_{m,b_{r^{k}}} - S_{m,b_{r^{k-1}}} - S_{r^{k} \lfloor \frac{m}{r^{k}} \rfloor,b_{r^{k}}} + S_{r^{k} \lfloor \frac{m}{r^{k}} \rfloor,b_{r^{k-1}}}).
\end{split}
\end{equation}
Further,
\begin{equation}\label{eq:2.4}
S_{m,b_{r^{k}}} - S_{m,b_{r^{k-1}}} = \sum_{j=1}^{m} [(X_{j,b_{r^{k}}}' - X_{j,b_{r^{k-1}}}') - \mathbb{E} (X_{j,b_{r^{k}}}' - X_{j,b_{r^{k-1}}}')]
\end{equation}
and
\begin{equation}\label{eq:2.5}
S_{r^{k} \lfloor \frac{m}{r^{k}} \rfloor,b_{r^{k}}} - S_{r^{k} \lfloor \frac{m}{r^{k}} \rfloor,b_{r^{k-1}}} = \sum_{j=1}^{r^{k} \lfloor \frac{m}{r^{k}} \rfloor} [(X_{j,b_{r^{k}}}' - X_{j,b_{r^{k-1}}}') - \mathbb{E} (X_{j,b_{r^{k}}}' - X_{j,b_{r^{k-1}}}')].
\end{equation}
Recall that
\begin{equation}\label{eq:2.6}
r^{k} \left\lfloor \frac{m}{r^{k}} \right\rfloor \leqslant m < r^{k} \left\lfloor \frac{m}{r^{k}} \right\rfloor + r^{k}, \qquad 1 \leqslant k \leqslant n + 1
\end{equation}
because $\lfloor t \rfloor \leqslant t < \lfloor t \rfloor + 1$ for any real number $t$ (see \cite{Epp20}). Thereby, \eqref{eq:2.4}, \eqref{eq:2.5} and \eqref{eq:2.6} entail
\begin{align}
&\lvert S_{m,b_{r^{k}}} - S_{m,b_{r^{k-1}}} - S_{r^{k} \lfloor \frac{m}{r^{k}} \rfloor,b_{r^{k}}} + S_{r^{k} \lfloor \frac{m}{r^{k}} \rfloor,b_{r^{k-1}}} \rvert \notag \\
&\quad = \Bigg\lvert\sum_{j=1 + r^{k} \lfloor \frac{m}{r^{k}} \rfloor}^{m} [(X_{j,b_{r^{k}}}' - X_{j,b_{r^{k-1}}}') - \mathbb{E} (X_{j,b_{r^{k}}}' - X_{j,b_{r^{k-1}}}')] \Bigg\rvert \notag \\
&\quad \leqslant \sum_{j=1 + r^{k} \lfloor \frac{m}{r^{k}} \rfloor}^{m} \lvert (X_{j,b_{r^{k}}}' - X_{j,b_{r^{k-1}}}') - \mathbb{E} (X_{j,b_{r^{k}}}' - X_{j,b_{r^{k-1}}}') \rvert \notag \\
&\quad \leqslant \sum_{j=1 + r^{k} \lfloor \frac{m}{r^{k}} \rfloor}^{r^{k} \lfloor \frac{m}{r^{k}} \rfloor + r^{k}} \lvert (X_{j,b_{r^{k}}}' - X_{j,b_{r^{k-1}}}') - \mathbb{E} (X_{j,b_{r^{k}}}' - X_{j,b_{r^{k-1}}}') \rvert \label{eq:2.7} \\
&\quad \leqslant \sum_{j=1 + r^{k} \lfloor \frac{m}{r^{k}} \rfloor}^{r^{k} \lfloor \frac{m}{r^{k}} \rfloor + r^{k}} (\lvert X_{j,b_{r^{k}}}' - X_{j,b_{r^{k-1}}}' \rvert + \mathbb{E} \lvert X_{j,b_{r^{k}}}' - X_{j,b_{r^{k-1}}}' \rvert) \notag \\
&\quad = \sum_{j=1 + r^{k} \lfloor \frac{m}{r^{k}} \rfloor}^{r^{k} \lfloor \frac{m}{r^{k}} \rfloor + r^{k}} (\lvert X_{j,b_{r^{k}}}' - X_{j,b_{r^{k-1}}}' \rvert - \mathbb{E} \lvert X_{j,b_{r^{k}}}' - X_{j,b_{r^{k-1}}}' \rvert) + 2 \sum_{j=1 + r^{k} \lfloor \frac{m}{r^{k}} \rfloor}^{r^{k} \lfloor \frac{m}{r^{k}} \rfloor + r^{k}} \mathbb{E} \lvert X_{j,b_{r^{k}}}' - X_{j,b_{r^{k-1}}}' \rvert \notag \\
&\quad \leqslant \Bigg\lvert\sum_{j=1 + r^{k} \lfloor \frac{m}{r^{k}} \rfloor}^{r^{k} \lfloor \frac{m}{r^{k}} \rfloor + r^{k}} (X_{j,b_{r^{k-1}},b_{r^{k}}}'' - \mathbb{E} X_{j,b_{r^{k-1}},b_{r^{k}}}'') \Bigg\rvert + 2 \sum_{j=1 + r^{k} \lfloor \frac{m}{r^{k}} \rfloor}^{r^{k} \lfloor \frac{m}{r^{k}} \rfloor + r^{k}} \mathbb{E} \lvert X_{j} \rvert I_{\left\{\lvert X_{j} \rvert > b_{r^{k-1}} \right\}} \notag
\end{align}
where
\begin{equation*}
X_{j,b_{r^{k-1}},b_{r^{k}}}'' := \lvert X_{j,b_{r^{k}}}' - X_{j,b_{r^{k-1}}}' \rvert = h_{b_{r^{k-1}},b_{r^{k}}}(X_{j}).
\end{equation*}
From \eqref{eq:2.6} we still get
\begin{equation*}
r \left\lfloor \frac{m}{r^{k}} \right\rfloor \leqslant \frac{m}{r^{k-1}} < r  \left\lfloor \frac{m}{r^{k}} \right\rfloor + r, \qquad 1 \leqslant k \leqslant n + 1
\end{equation*}
whence
\begin{equation*}
r \left\lfloor \frac{m}{r^{k}} \right\rfloor \leqslant \left\lfloor \frac{m}{r^{k-1}} \right\rfloor \leqslant r \left\lfloor \frac{m}{r^{k}} \right\rfloor + (r - 1),
\end{equation*}
i.e.
\begin{equation*}
r^{k-1} \left\lfloor \frac{m}{r^{k-1}} \right\rfloor \in \left\{r^{k} \left\lfloor \frac{m}{r^{k}} \right\rfloor, r^{k} \left\lfloor \frac{m}{r^{k}} \right\rfloor + r^{k-1}, r^{k} \left\lfloor \frac{m}{r^{k}} \right\rfloor + 2r^{k-1}, \ldots, r^{k} \left\lfloor \frac{m}{r^{k}} \right\rfloor + (r - 1)r^{k-1} \right\}
\end{equation*}
for every $1 \leqslant k \leqslant n+1$, $1 \leqslant m < r^{n+1}$ $(n \geqslant 0)$. Thus,
\begin{align}
\lvert S_{r^{k-1} \lfloor \frac{m}{r^{k-1}} \rfloor,b_{r^{k-1}}} - S_{r^{k} \lfloor \frac{m}{r^{k}} \rfloor,b_{r^{k-1}}} \rvert &= \Bigg\lvert \sum_{j=1 + r^{k} \lfloor \frac{m}{r^{k}} \rfloor}^{r^{k-1} \lfloor \frac{m}{r^{k-1}} \rfloor} (X_{j,b_{r^{k-1}}}' - \mathbb{E} X_{j,b_{r^{k-1}}}')\Bigg\rvert \notag \\
&\leqslant \max_{0 \leqslant \ell \leqslant r-1} \Bigg\lvert \sum_{j=1 + r^{k} \lfloor \frac{m}{r^{k}} \rfloor}^{r^{k} \lfloor \frac{m}{r^{k}} \rfloor + \ell r^{k-1}} (X_{j,b_{r^{k-1}}}' - \mathbb{E} X_{j,b_{r^{k-1}}}')\Bigg\rvert \label{eq:2.8} \\
&= \max_{1 \leqslant \ell \leqslant r-1} \Bigg\lvert \sum_{j=1 + r^{k} \lfloor \frac{m}{r^{k}} \rfloor}^{r^{k} \lfloor \frac{m}{r^{k}} \rfloor + \ell r^{k-1}} (X_{j,b_{r^{k-1}}}' - \mathbb{E} X_{j,b_{r^{k-1}}}')\Bigg\rvert. \notag
\end{align}
By \eqref{eq:2.3}, \eqref{eq:2.7} and \eqref{eq:2.8}, it follows
\begin{align*}
\lvert S_{m,b_{r^{n+1}}} \rvert &\leqslant \sum_{k=1}^{n+1} \max_{1 \leqslant \ell \leqslant r-1} \Bigg\lvert \sum_{j=1 + r^{k} \lfloor \frac{m}{r^{k}} \rfloor}^{r^{k} \lfloor \frac{m}{r^{k}} \rfloor + \ell r^{k-1}} (X_{j,b_{r^{k-1}}}' - \mathbb{E} X_{j,b_{r^{k-1}}}') \Bigg\rvert \\
& \quad + \sum_{k=1}^{n+1} \Bigg\lvert \sum_{j=1 + r^{k} \lfloor \frac{m}{r^{k}} \rfloor}^{r^{k} \lfloor \frac{m}{r^{k}} \rfloor + r^{k}} (X_{j,b_{r^{k-1}},b_{r^{k}}}'' - \mathbb{E} X_{j,b_{r^{k-1}},b_{r^{k}}}'') \Bigg\rvert + 2 \sum_{k=1}^{n+1} \sum_{j=1 + r^{k} \lfloor \frac{m}{r^{k}} \rfloor}^{r^{k} \lfloor \frac{m}{r^{k}} \rfloor + r^{k}} \mathbb{E} \lvert X_{j} \rvert I_{\left\{\lvert X_{j} \rvert > b_{r^{k-1}} \right\}}
\end{align*}
so that
\begin{align}
\max_{1 \leqslant m < r^{n+1}} \lvert S_{m,b_{r^{n+1}}} \rvert & \leqslant \sum_{k=1}^{n+1} \max_{1 \leqslant m < r^{n+1}} \left[\max_{1 \leqslant \ell \leqslant r-1} \Bigg\lvert \sum_{j=1 + r^{k} \lfloor \frac{m}{r^{k}} \rfloor}^{r^{k} \lfloor \frac{m}{r^{k}} \rfloor + \ell r^{k-1}} (X_{j,b_{r^{k-1}}}' - \mathbb{E} X_{j,b_{r^{k-1}}}') \Bigg\rvert \right] \notag \\
& \qquad + \sum_{k=1}^{n+1} \max_{1 \leqslant m < r^{n+1}} \Bigg\lvert \sum_{j=1 + r^{k} \lfloor \frac{m}{r^{k}} \rfloor}^{r^{k} \lfloor \frac{m}{r^{k}} \rfloor + r^{k}} (X_{j,b_{r^{k-1}},b_{r^{k}}}'' - \mathbb{E} X_{j,b_{r^{k-1}},b_{r^{k}}}'') \Bigg\rvert \notag \\
& \qquad + 2 \sum_{k=1}^{n+1} \max_{1 \leqslant m < r^{n+1}} \sum_{j=1 + r^{k} \lfloor \frac{m}{r^{k}} \rfloor}^{r^{k} \lfloor \frac{m}{r^{k}} \rfloor + r^{k}} \mathbb{E} \lvert X_{j} \rvert I_{\left\{\lvert X_{j} \rvert > b_{r^{k-1}} \right\}} \label{eq:2.9} \\
&\leqslant \sum_{k=1}^{n+1} \max_{1 \leqslant \ell \leqslant r-1} \left[\max_{0 \leqslant h < r^{n-k+1}} \Bigg\lvert \sum_{j=1 + h r^{k}}^{h r^{k} + \ell r^{k-1}} (X_{j,b_{r^{k-1}}}' - \mathbb{E} X_{j,b_{r^{k-1}}}') \Bigg\rvert \right] \notag \\
& \qquad + \sum_{k=1}^{n+1} \max_{0 \leqslant h < r^{n-k+1}} \Bigg\lvert \sum_{j=1 + h r^{k}}^{h r^{k} + r^{k}} (X_{j,b_{r^{k-1}},b_{r^{k}}}'' - \mathbb{E} X_{j,b_{r^{k-1}},b_{r^{k}}}'') \Bigg\rvert \notag \\
&\qquad + 2 \sum_{k=1}^{n+1} \max_{0 \leqslant h < r^{n-k+1}} \sum_{j=1 + h r^{k}}^{h r^{k} + r^{k}} \mathbb{E} \lvert X_{j} \rvert I_{\left\{\lvert X_{j} \rvert > b_{r^{k-1}} \right\}} \notag
\end{align}
because
\begin{equation*}
0 = \left\lfloor \frac{1}{r^{k}} \right\rfloor \leqslant \left\lfloor \frac{m}{r^{k}} \right\rfloor \leqslant \frac{m}{r^{k}} < r^{n-k+1} \quad \Longrightarrow \quad 0 \leqslant \left\lfloor \frac{m}{r^{k}} \right\rfloor \leqslant r^{n-k+1} - 1
\end{equation*}
for all $1 \leqslant k \leqslant n+1$, $1 \leqslant m < r^{n+1}$ $(n \geqslant 0)$. By assumption (b), there is a positive integer $n_{0}$ such that
\begin{equation}\label{eq:2.10}
\sum_{k=1}^{n+1} \max_{0 \leqslant h < r^{n-k+1}} \sum_{j=1 + h r^{k}}^{h r^{k} + r^{k}} \mathbb{E} \lvert X_{j} \rvert I_{\left\{\lvert X_{j} \rvert > b_{r^{k-1}} \right\}} < \frac{\varepsilon b_{r^{n}}}{4}
\end{equation}
for all $n \geqslant n_{0}$. Combining \eqref{eq:2.9} and \eqref{eq:2.10} we obtain, for every $n \geqslant n_{0}$,
\begin{align*}
&\mathbb{P} \left\{\max_{1 \leqslant m < r^{n+1}} \lvert S_{m,b_{r^{n+1}}} \rvert > \varepsilon b_{r^{n}} \right\} \\
&\quad \leqslant \mathbb{P} \left\{\sum_{k=1}^{n+1} \max_{1 \leqslant \ell \leqslant r-1} \Bigg[\max_{0 \leqslant h < r^{n-k+1}} \Bigg\lvert \sum_{j=1 + h r^{k}}^{h r^{k} + \ell r^{k-1}} (X_{j,b_{r^{k-1}}}' - \mathbb{E} X_{j,b_{r^{k-1}}}') \Bigg\rvert \Bigg] \right. \\
&\qquad \left. + \sum_{k=1}^{n+1} \max_{0 \leqslant h < r^{n-k+1}} \Bigg\lvert \sum_{j=1 + h r^{k}}^{h r^{k} + r^{k}} (X_{j,b_{r^{k-1}},b_{r^{k}}}'' - \mathbb{E} X_{j,b_{r^{k-1}},b_{r^{k}}}'') \Bigg\rvert > \frac{\varepsilon b_{r^{n}}}{2} \right\} \\
&\quad \leqslant \mathbb{P} \left\{\sum_{k=1}^{n+1} \max_{1 \leqslant \ell \leqslant r-1} \Bigg[\max_{0 \leqslant h < r^{n-k+1}} \Bigg\lvert \sum_{j=1 + h r^{k}}^{h r^{k} + \ell r^{k-1}} (X_{j,b_{r^{k-1}}}' - \mathbb{E} X_{j,b_{r^{k-1}}}') \Bigg\rvert \Bigg] > \frac{\varepsilon b_{r^{n}}}{4} \right\} \\
&\qquad + \mathbb{P} \left\{\sum_{k=1}^{n+1} \max_{0 \leqslant h < r^{n-k+1}} \Bigg\lvert \sum_{j=1 + h r^{k}}^{h r^{k} + r^{k}} (X_{j,b_{r^{k-1}},b_{r^{k}}}'' - \mathbb{E} X_{j,b_{r^{k-1}},b_{r^{k}}}'') \Bigg\rvert > \frac{\varepsilon b_{r^{n}}}{4} \right\}.
\end{align*}
The thesis holds by noting that
\begin{align*}
&\mathbb{P} \left\{\sum_{k=1}^{n+1} \max_{1 \leqslant \ell \leqslant r-1} \Bigg[\max_{0 \leqslant h < r^{n-k+1}} \Bigg\lvert \sum_{j=1 + h r^{k}}^{h r^{k} + \ell r^{k-1}} (X_{j,b_{r^{k-1}}}' - \mathbb{E} X_{j,b_{r^{k-1}}}') \Bigg\rvert \Bigg] > \frac{\varepsilon b_{r^{n}}}{4} \right\} \\
&\quad \leqslant \mathbb{P} \left\{\sum_{k=1}^{n+1} \max_{1 \leqslant \ell \leqslant r-1} \Bigg[\max_{0 \leqslant h < r^{n-k+1}} \Bigg\lvert \sum_{j=1 + h r^{k}}^{h r^{k} + \ell r^{k-1}} (X_{j,b_{r^{k-1}}}' - \mathbb{E} X_{j,b_{r^{k-1}}}') \Bigg\rvert \Bigg] > \frac{\varepsilon}{C} \sum_{k=1}^{n+1} a_{n,k} \right\} \\
&\quad \leqslant \sum_{k=1}^{n+1} \mathbb{P} \left\{\max_{1 \leqslant \ell \leqslant r-1} \Bigg[\max_{0 \leqslant h < r^{n-k+1}} \Bigg\lvert \sum_{j=1 + h r^{k}}^{h r^{k} + \ell r^{k-1}} (X_{j,b_{r^{k-1}}}' - \mathbb{E} X_{j,b_{r^{k-1}}}') \Bigg\rvert \Bigg] > \frac{\varepsilon a_{n,k}}{C} \right\} \\
&\quad \leqslant \frac{C}{\varepsilon^{2}} \sum_{k=1}^{n+1} \frac{1}{a_{n,k}^{2}} \mathbb{E} \left\{\max_{1 \leqslant \ell \leqslant r - 1} \left[\max_{0 \leqslant h < r^{n-k+1}} \Bigg\lvert \sum_{j=1 + h r^{k}}^{h r^{k} + \ell r^{k-1}} (X_{j,b_{r^{k-1}}}' - \mathbb{E} X_{j,b_{r^{k-1}}}') \Bigg\rvert \right] \right\}^{2} \\
&\quad = \frac{C}{\varepsilon^{2}} \sum_{k=1}^{n+1} \frac{1}{a_{n,k}^{2}} \mathbb{E} \max_{1 \leqslant \ell \leqslant r - 1} \left\{\max_{0 \leqslant h < r^{n-k+1}} \left[\sum_{j=1 + h r^{k}}^{h r^{k} + \ell r^{k-1}}  (X_{j,b_{r^{k-1}}}' - \mathbb{E} X_{j,b_{r^{k-1}}}') \right]^{2} \right\} \\
&\quad \leqslant \frac{C}{\varepsilon^{2}} \sum_{k=1}^{n+1} \frac{1}{a_{n,k}^{2}} \mathbb{E} \sum_{\ell = 1}^{r - 1} \left\{\sum_{h=0}^{r^{n-k+1} - 1} \left[\sum_{j=1 + h r^{k}}^{h r^{k} + \ell r^{k-1}} (X_{j,b_{r^{k-1}}}' - \mathbb{E} X_{j,b_{r^{k-1}}}') \right]^{2} \right\} \\
&\quad \leqslant \frac{C}{\varepsilon^{2}} \sum_{k=1}^{n+1} \frac{1}{a_{n,k}^{2}} \sum_{\ell=1}^{r-1} \sum_{h=0}^{r^{n-k+1} - 1} \sum_{j=1 + h r^{k}}^{h r^{k} + \ell r^{k-1}} \mathbb{E} \lvert X_{j,b_{r^{k-1}}}' \rvert^{2} \\
&\qquad + \frac{C}{\varepsilon^{2}} \sum_{k=1}^{n+1} \frac{1}{a_{n,k}^{2}} \sum_{\ell=1}^{r-1} \sum_{h=0}^{r^{n-k+1} - 1} \sum_{1 + h r^{k} \leqslant i < j \leqslant h r^{k} + \ell r^{k-1}} \mathrm{Cov}(X_{i,b_{r^{k-1}}}';X_{j,b_{r^{k-1}}}') \\
&\quad = \frac{C}{\varepsilon^{2}} \sum_{k=1}^{n+1}\frac{1}{a_{n,k}^{2}} \sum_{\ell=1}^{r-1} \sum_{h=0}^{r^{n-k+1} - 1} \sum_{j=1 + h r^{k}}^{h r^{k} + \ell r^{k-1}} \left(\mathbb{E} X_{j}^{2} I_{\left\{\lvert X_{j} \rvert \leqslant b_{r^{k-1}} \right\}} + b_{r^{k-1}}^{2} \mathbb{P} \big\{\lvert X_{j} \rvert > b_{r^{k-1}} \big\} \right) \\
&\qquad + \frac{C}{\varepsilon^{2}} \sum_{k=1}^{n+1} \frac{1}{a_{n,k}^{2}} \sum_{\ell=1}^{r-1} \sum_{h=0}^{r^{n-k+1} - 1} \sum_{1 + h r^{k} \leqslant i < j \leqslant h r^{k} + \ell r^{k-1}} G_{X_{i},X_{j}}(b_{r^{k-1}}) \\
&\quad \leqslant \frac{C(r)}{\varepsilon^{2}} \sum_{k=1}^{n+1} \frac{1}{a_{n,k}^{2}}  \sum_{h=0}^{r^{n-k+1} - 1} \sum_{j=1 + h r^{k}}^{h r^{k} + r^{k}} \left(\mathbb{E} X_{j}^{2} I_{\left\{\lvert X_{j} \rvert \leqslant b_{r^{k-1}} \right\}} + b_{r^{k-1}}^{2} \mathbb{P} \big\{\lvert X_{j} \rvert > b_{r^{k-1}} \big\} \right) \\
&\qquad + \frac{C}{\varepsilon^{2}} \sum_{k=1}^{n+1} \frac{1}{a_{n,k}^{2}} \sum_{\ell=1}^{r-1} \sum_{h=0}^{r^{n-k+1} - 1} \left[\sum_{1 + h r^{k} \leqslant i < j \leqslant h r^{k} + \ell r^{k-1}} G_{X_{i},X_{j}}(b_{r^{k-1}}) \right]^{+} \\
&\quad \leqslant \frac{C(r)}{\varepsilon^{2}} \sum_{k=1}^{n+1} \sum_{j=1}^{r^{n+1}} \frac{1}{a_{n,k}^{2}} \mathbb{E} X_{j}^{2} I_{\left\{\lvert X_{j} \rvert \leqslant b_{r^{k-1}} \right\}} + \frac{C(r)}{\varepsilon^{2}} \sum_{k=1}^{n+1} \sum_{j=1}^{r^{n+1}} \frac{b_{r^{k-1}}^{2} }{a_{n,k}^{2}} \mathbb{P} \left\{\lvert X_{j} \rvert > b_{r^{k-1}} \right\} \\
&\qquad + \frac{C(r)}{\varepsilon^{2}} \sum_{k=1}^{n+1} \frac{1}{a_{n,k}^{2}} \max_{1 \leqslant \ell \leqslant r - 1} \sum_{h=0}^{r^{n-k+1} - 1} \left[\sum_{1 + h r^{k} \leqslant i < j \leqslant h r^{k} + \ell r^{k-1}} G_{X_{i},X_{j}}(b_{r^{k-1}}) \right]^{+}
\end{align*}
and
\begin{align*}
&\mathbb{P} \left\{\sum_{k=1}^{n+1} \max_{0 \leqslant h < r^{n-k+1}} \Bigg\lvert \sum_{j=1 + h r^{k}}^{h r^{k} + r^{k}} (X_{j,b_{r^{k-1}},b_{r^{k}}}'' - \mathbb{E} X_{j,b_{r^{k-1}},b_{r^{k}}}'') \Bigg\rvert > \frac{\varepsilon b_{r^{n}}}{4} \right\} \\
&\quad \leqslant \mathbb{P} \left\{\sum_{k=1}^{n+1} \max_{0 \leqslant h < r^{n-k+1}} \Bigg\lvert \sum_{j=1 + h r^{k}}^{h r^{k} + r^{k}} (X_{j,b_{r^{k-1}},b_{r^{k}}}'' - \mathbb{E} X_{j,b_{r^{k-1}},b_{r^{k}}}'') \Bigg\rvert > \frac{\varepsilon}{C} \sum_{k=1}^{n+1} a_{n,k} \right\} \\
&\quad \leqslant \sum_{k=1}^{n+1} \mathbb{P} \left\{\max_{0 \leqslant h < r^{n-k+1}} \Bigg\lvert \sum_{j=1 + h r^{k}}^{h r^{k} + r^{k}} (X_{j,b_{r^{k-1}},b_{r^{k}}}'' - \mathbb{E} X_{j,b_{r^{k-1}},b_{r^{k}}}'') \Bigg\rvert > \frac{\varepsilon a_{n,k}}{C} \right\} \\
&\quad \leqslant \frac{C}{\varepsilon^{2}} \sum_{k=1}^{n+1} \frac{1}{a_{n,k}^{2}} \mathbb{E} \left[\max_{0 \leqslant h < r^{n-k+1}} \Bigg\lvert \sum_{j=1 + h r^{k}}^{h r^{k} + r^{k}} (X_{j,b_{r^{k-1}},b_{r^{k}}}'' - \mathbb{E} X_{j,b_{r^{k-1}},b_{r^{k}}}'') \Bigg\rvert \right]^{2} \\
&\quad = \frac{C}{\varepsilon^{2}} \sum_{k=1}^{n+1} \frac{1}{a_{n,k}^{2}} \mathbb{E} \max_{0 \leqslant h < r^{n-k+1}} \left[\sum_{j=1 + h r^{k}}^{h r^{k} + r^{k}} (X_{j,b_{r^{k-1}},b_{r^{k}}}'' - \mathbb{E} X_{j,b_{r^{k-1}},b_{r^{k}}}'') \right]^{2} \\
&\quad \leqslant \frac{C}{\varepsilon^{2}} \sum_{j=1}^{k+1} \frac{1}{a_{j,r^{k}}^{2}} \sum_{h=0}^{r^{n-k+1} - 1} \mathbb{E} \left[ \sum_{j=1 + h r^{k}}^{h r^{k} + r^{k}} (X_{j,b_{r^{k-1}},b_{r^{k}}}'' - \mathbb{E} X_{j,b_{r^{k-1}},b_{r^{k}}}'') \right]^{2} \\
&\quad \leqslant \frac{C}{\varepsilon^{2}} \sum_{k=1}^{n+1} \frac{1}{a_{n,k}^{2}}  \sum_{h=0}^{r^{n-k+1} - 1} \sum_{j=1 + h r^{k}}^{h r^{k} + r^{k}} \mathbb{E} (X_{j,b_{r^{k-1}},b_{r^{k}}}'' - \mathbb{E} X_{j,b_{r^{k-1}},b_{r^{k}}}'')^{2} \\
&\qquad + \frac{C}{\varepsilon^{2}} \sum_{k=1}^{n+1} \frac{1}{a_{n,k}^{2}} \sum_{h=0}^{r^{n-k+1} - 1} \sum_{1 + h r^{k} \leqslant i < j \leqslant h r^{k} + r^{k}} \mathrm{Cov}(X_{j,b_{r^{k-1}},b_{r^{k}}}'' ; X_{j,b_{r^{k-1}},b_{r^{k}}}'') \\
&\quad \leqslant \frac{C}{\varepsilon^{2}} \sum_{k=1}^{n+1} \frac{1}{a_{n,k}^{2}}  \sum_{h=0}^{r^{n-k+1} - 1} \sum_{j=1 + h r^{k}}^{h r^{k} + r^{k}} \mathbb{E}( X_{j,b_{r^{k-1}},b_{r^{k}}}'')^{2} \\
&\qquad + \frac{C}{\varepsilon^{2}} \sum_{k=1}^{n+1} \frac{1}{a_{n,k}^{2}} \sum_{h=0}^{r^{n-k+1} - 1} \sum_{1 + h r^{k} \leqslant i < j \leqslant h r^{k} + r^{k}} H_{X_{i},X_{j}}(b_{r^{k-1}},b_{r^{k}}) \\
&\quad \leqslant \frac{C}{\varepsilon^{2}} \sum_{k=1}^{n+1} \frac{1}{a_{n,k}^{2}}  \sum_{h=0}^{r^{n-k+1} - 1} \sum_{j=1 + h r^{k}}^{h r^{k} + r^{k}} \left(\mathbb{E} X_{j}^{2} I_{\left\{b_{r^{k-1}} < \lvert X_{j} \rvert \leqslant b_{r^{k}} \right\}} + b_{r^{k}}^{2} \mathbb{P} \big\{\lvert X_{j} \rvert > b_{r^{k}} \big\} \right) \\
&\qquad + \frac{C}{\varepsilon^{2}} \sum_{k=1}^{n+1} \frac{1}{a_{n,k}^{2}}  \sum_{h=0}^{r^{n-k+1} - 1} \left[\sum_{1 + h r^{k} \leqslant i < j \leqslant h r^{k} + r^{k}} H_{X_{i},X_{j}}(b_{r^{k-1}},b_{r^{k}}) \right]^{+} \\
&\quad \leqslant \frac{C}{\varepsilon^{2}} \sum_{k=1}^{n+1}  \sum_{j=1}^{r^{n+1}} \frac{1}{a_{n,k}^{2}} \mathbb{E} X_{j}^{2} I_{\left\{b_{r^{k-1}} < \lvert X_{j} \rvert \leqslant b_{r^{k}} \right\}} + \frac{C}{\varepsilon^{2}} \sum_{k=1}^{n+1}  \sum_{j=1}^{r^{n+1}} \frac{b_{r^{k}}^{2}}{a_{n,k}^{2}} \mathbb{P} \big\{\lvert X_{j} \rvert > b_{r^{k}} \big\} \\
&\qquad + \frac{C}{\varepsilon^{2}} \sum_{k=1}^{n+1} \frac{1}{a_{n,k}^{2}} \sum_{h=0}^{r^{n-k+1} - 1} \left[\sum_{1 + h r^{k} \leqslant i < j \leqslant h r^{k} + r^{k}} H_{X_{i},X_{j}}(b_{r^{k-1}},b_{r^{k}}) \right]^{+}
\end{align*}
for each $n \geqslant n_{0}$.
\end{proof}

A sequence $\{X_{n}, \, n \geqslant 1 \}$ of random variables is \emph{stochastically dominated} by a random variable $X$ if there exists a constant $C>0$ such that $\sup_{n \geqslant 1} \mathbb{P} \left\{\abs{X_{n}} > t \right\} \leqslant C \, \mathbb{P} \left\{\abs{X} > t  \right\}$ for all $t > 0$. Under this extra assumption on the random sequence $\{X_{n}, \, n \geqslant 1 \}$ and employing Lemma 1 of \cite{Lita15}, Theorem~\ref{thr:2.1} can be restated as follows:

\begin{corollary}\label{cor:2.1}
Let $r > 1$ be an integer and $\left\{X_{n}, \, n \geqslant 1 \right\}$ a sequence of random variables stochastically dominated by a random variable $X \in \mathscr{L}_{1}$. If $\{b_{n} \}$ is nondecreasing sequence of positive constants, and $\{a_{n,k}, \, 1 \leqslant k \leqslant n + 1, n \geqslant 1 \}$ is an array of positive constants satisfying condition \textnormal{(a)} of Theorem~\ref{thr:2.1} and \\

\noindent \textnormal{(b')} ${\displaystyle \sum_{k=1}^{n+1} r^{k} \mathbb{E} \lvert X \rvert I_{\left\{\lvert X \rvert > b_{r^{k-1}} \right\}} = o(b_{r^{n}})}$ as $n \rightarrow \infty$ \\

\noindent then, for any $\varepsilon > 0$, there is a positive integer $n_{0}$ and a constant $C(r) > 0$ such that
\begin{equation*}
\begin{split}
    &\mathbb{P} \left\{\max_{1 \leqslant m < r^{n+1}} \Bigg\lvert\sum_{k=1}^{m} (X_{k} - \mathbb{E} \, X_{k}) \Bigg\rvert > \varepsilon b_{r^{n}} \right\} \\
    &\quad \leqslant C(r) r^{n} \mathbb{P} \left\{\lvert X \rvert > b_{r^{n+1}} \right\} + \frac{C(r) r^{n}}{\varepsilon^{2}} \sum_{k=1}^{n+1} a_{n,k}^{-2} \mathbb{E} X^{2} I_{\left\{\lvert X \rvert \leqslant b_{r^{k}} \right\}} + \frac{C(r) r^{n} }{\varepsilon^{2}} \sum_{k=1}^{n+1} \frac{b_{r^{k}}^{2}}{a_{n,k}^{2}} \mathbb{P} \big\{\lvert X \rvert > b_{r^{k-1}} \big\} \\
    &\qquad + \frac{C(r)}{\varepsilon^{2}} \sum_{k=1}^{n+1} a_{n,k}^{-2} \max_{1 \leqslant \ell \leqslant r-1} \sum_{h=0}^{r^{n-k+1} - 1} \left[\sum_{1 + hr^{k} \leqslant i < j \leqslant h r^{k} + \ell r^{k-1}}G_{X_{i},X_{j}}(b_{r^{k-1}}) \right]^{+} \\
    &\qquad + \frac{C(r)}{\varepsilon^{2}} \sum_{k=1}^{n+1} a_{n,k}^{-2} \sum_{h=0}^{r^{n-k+1} - 1} \left[\sum_{1 + hr^{k} \leqslant i < j \leqslant h r^{k} + r^{k}} H_{X_{i},X_{j}}(b_{r^{k-1}},b_{r^{k}}) \right]^{+}
\end{split}
\end{equation*}
for all $n \geqslant n_{0}$.
\end{corollary}

\section{SLLN for dependent random variables}

\indent

The next result allows us to obtain SLLN for sequences $\left\{X_{n}, \, n \geqslant 1 \right\}$ of random variables stochastically dominated by a random variable $X \in \mathscr{L}_{1}$. Let us stress out that no assumption concerning to dependence is assumed for $\{X_{n}, \, n \geqslant 1 \}$.

\begin{theorem}\label{thr:3.1}
Let $r > 1$ be an integer and $\{X_{n}, \, n \geqslant 1 \}$ be a sequence of random variables stochastically dominated by a random variable $X \in \mathscr{L}_{1}$. If $\{b_{n} \}$ is a nondecreasing sequence of positive constants, $\{c_{n} \}$ a nonincreasing sequence of nonnegative constants, $\{a_{n,k}, \, 1 \leqslant k \leqslant n + 1, n \geqslant 1 \}$ is an array of positive constants satisfying assumptions $\mathrm{(a)}$, $\mathrm{(b')}$, and \\

\noindent \textnormal{(c)} ${\displaystyle\sum_{m=0}^{\infty} \sum_{n=r^{m}}^{r^{m+1} - 1} r^{m} c_{n} \mathbb{P} \left\{\lvert X \rvert > b_{r^{m+1}} \right\} < \infty}$, \\

\noindent \textnormal{(d)} ${\displaystyle\sum_{m=0}^{\infty} \sum_{n=r^{m}}^{r^{m+1} - 1} \sum_{k=1}^{m+1}\frac{c_{n} r^{m}}{a_{m,k}^{2}} \mathbb{E} X^{2} I_{\left\{\lvert X \rvert \leqslant b_{r^{k}} \right\}} < \infty}$, \\

\noindent \textnormal{(e)} ${\displaystyle\sum_{m=0}^{\infty} \sum_{n=r^{m}}^{r^{m+1} - 1} \sum_{k=1}^{m+1} \frac{c_{n} r^{m} b_{r^{k}}^{2}}{a_{m,k}^{2}} \mathbb{P} \big\{\lvert X \rvert > b_{r^{k-1}} \big\} < \infty}$, \\

\noindent \textnormal{(f)} ${\displaystyle\sum_{m=0}^{\infty} \sum_{n=r^{m}}^{r^{m+1} - 1} \sum_{k=1}^{m+1} \frac{c_{n}}{a_{m,k}^{2}} \max_{1 \leqslant \ell \leqslant r-1} \sum_{h=0}^{r^{m-k+1} - 1} \left[\sum_{1 + hr^{k} \leqslant i < j \leqslant h r^{k} + \ell r^{k-1}} G_{X_{i},X_{j}}(b_{r^{k-1}}) \right]^{+} < \infty}$, \\

\noindent \textnormal{(g)} ${\displaystyle\sum_{m=0}^{\infty} \sum_{n=r^{m}}^{r^{m+1} - 1} \sum_{k=1}^{m+1} \frac{c_{n}}{a_{m,k}^{2}} \sum_{h=0}^{r^{m-k+1} - 1} \left[\sum_{1 + hr^{k} \leqslant i < j \leqslant h r^{k} + r^{k}} H_{X_{i},X_{j}}(b_{r^{k-1}},b_{r^{k}}) \right]^{+} < \infty}$, \\

\noindent then
\begin{equation}
\sum_{n=1}^{\infty} c_{n} \mathbb{P} \left\{\max_{1 \leqslant j \leqslant n} \abs{\sum_{k=1}^{j}(X_{k} - \mathbb{E} \, X_{k})} > \varepsilon b_{n} \right\} < \infty
\end{equation}
for all $\varepsilon > 0$.
\end{theorem}

\begin{proof}
Let $r > 1$ be a (fixed) integer and $\varepsilon > 0$. Hence,
\begin{align*}
&\sum_{n=1}^{\infty} c_{n} \mathbb{P} \left\{\max_{1 \leqslant j \leqslant n} \abs{\sum_{k=1}^{j}(X_{k} - \mathbb{E} \, X_{k})} > \varepsilon b_{n} \right\} \\
&\qquad = \sum_{m=0}^{\infty} \sum_{n=r^{m}}^{r^{m+1} - 1} c_{n} \mathbb{P} \left\{\max_{1 \leqslant j \leqslant n} \abs{\sum_{k=1}^{j}(X_{k} - \mathbb{E} \, X_{k})} > \varepsilon b_{n} \right\} \\
&\qquad \leqslant \sum_{m=0}^{\infty} \mathbb{P} \left\{\max_{1 \leqslant j < r^{m+1}} \abs{\sum_{k=1}^{j}(X_{k} - \mathbb{E} \, X_{k})} > \varepsilon b_{r^{m}} \right\} \sum_{n=r^{m}}^{r^{m+1} - 1} c_{n}
\end{align*}
and the conclusion is a direct consequence of Corollary~\ref{cor:2.1}.
\end{proof}

\begin{corollary}\label{cor:3.1}
Let $r > 1$ be an integer, $1 \leqslant p < 2$, and $\{X_{n}, \, n \geqslant 1 \}$ be a sequence of random variables stochastically dominated by a random variable $X \in \mathscr{L}_{p}$. If
\begin{equation}\label{eq:3.2}
\sum_{1 \leqslant i < j \leqslant \infty} \sum_{k=1}^{\infty} r^{(2/\alpha - 2/p)k - 2(k \vee \log_{r} j)/\alpha} \left[G_{X_{i},X_{j}}^{+} \big(r^{(k - 1)/p} \big) + H_{X_{i},X_{j}}^{+} \big(r^{(k - 1)/p}, r^{k/p} \big) \right] < \infty
\end{equation}
for some $p < \alpha < 2$, then $\sum_{k = 1}^{n} (X_{k} - \mathbb{E} \, X_{k})/n^{1/p} \overset{\textnormal{a.s.}}{\longrightarrow} 0$.
\end{corollary}

\begin{proof}
Consider $1 \leqslant p < \alpha < 2$, a (fixed) integer $r > 1$, $c_{n} = 1/n$, $b_{n} = n^{1/p}$, and $a_{n,k} = r^{n/\alpha + (1/p - 1/\alpha) k}$. Hence,
\begin{gather*}
\limsup_{m \rightarrow \infty} \frac{b_{r^{m+1}}}{b_{r^{m}}} = \limsup_{m \rightarrow \infty} \frac{r^{(m+1)/p}}{r^{m/p}} = r^{1/p}, \\
\liminf_{m \rightarrow \infty} \sum_{n=r^{m}}^{r^{m+1} - 1} c_{n} \geqslant \liminf_{m \rightarrow \infty} \int_{r^{m}}^{r^{m+1}} \frac{\mathrm{d}u}{u} = \log r
\end{gather*}
and from Theorem 2.2 of \cite{Hu16}, it suffices to show
\begin{equation}\label{eq:3.3}
\sum_{n=1}^{\infty} \frac{1}{n} \mathbb{P} \left\{\max_{1 \leqslant j \leqslant n} \abs{\sum_{k=1}^{j}(X_{k} - \mathbb{E} \, X_{k})} > \varepsilon n^{1/p} \right\} < \infty
\end{equation}
for all $\varepsilon > 0$. We have
\begin{equation*}
\sum_{k=1}^{n+1} a_{n,k} = r^{n/\alpha} \cdot \frac{r^{n/p - n/\alpha + 2/p - 2/\alpha} - r^{1/p - 1/\alpha}}{r^{1/p - 1/\alpha} - 1} \leqslant \frac{r^{2/p - 2/\alpha}}{r^{1/p - 1/\alpha} - 1} \cdot r^{n/p}
\end{equation*}
so that assumption (a) of Theorem~\ref{thr:3.1} holds. On the other hand, the dominated convergence theorem guarantees
\begin{equation*}
\mathbb{E} \lvert X \rvert^{p} I_{\left\{\lvert X \rvert > r^{(n-1)/p} \right\}} = o(1), \quad n \rightarrow \infty.
\end{equation*}
Moreover, for each $k \geqslant 1$,
\begin{equation*}
r^{k + (1 - p)(k - 1)/p - n/p} = o(1), \quad n \rightarrow \infty
\end{equation*}
and
\begin{align*}
\sum_{k=1}^{n+1} r^{k + (1 - p)(k - 1)/p - n/p} = \frac{r^{1 + 1/p} - r^{1 - n/p}}{r^{1/p} - 1} \leqslant \frac{r^{1 + 1/p}}{r^{1/p} - 1}
\end{align*}
which implies
\begin{align*}
\sum_{k=1}^{n+1} r^{k - n/p} \mathbb{E} \lvert X \rvert I_{\left\{\lvert X \rvert > r^{(k-1)/p} \right\}} &\leqslant \sum_{k=1}^{n+1} r^{k + (1-p)(k-1)/p - n/p} \mathbb{E} \lvert X \rvert^{p} I_{\left\{\lvert X \rvert > r^{(k-1)/p} \right\}} = o(1), \quad n \rightarrow \infty
\end{align*}
by Toeplitz's lemma (see \cite{Linero13}) proving assumption (b') of Theorem~\ref{thr:3.1}. According to
\begin{equation}\label{eq:3.4}
\sum_{n=r^{m}}^{r^{m+1} - 1} c_{n} \leqslant \frac{r^{m+1} - r^{m}}{r^{m}} = r - 1
\end{equation}
one still get
\begin{align*}
\sum_{m=0}^{\infty} \sum_{n=r^{m}}^{r^{m+1} - 1} r^{m} c_{n} \mathbb{P} \left\{\lvert X \rvert > b_{r^{m+1}} \right\} &\leqslant (r - 1) \sum_{m=0}^{\infty} r^{m} \mathbb{P} \left\{\lvert X \rvert > r^{(m+1)/p} \right\} \\
&= (r - 1) \sum_{m=0}^{\infty} \sum_{j=m}^{\infty} r^{m} \mathbb{P} \left\{r^{j+1} < \lvert X \rvert^{p} \leqslant r^{j+2} \right\} \\
&= (r - 1) \sum_{j=0}^{\infty} \sum_{m=0}^{j} r^{m} \mathbb{P} \left\{r^{j+1} < \lvert X \rvert^{p} \leqslant r^{j+2} \right\} \\
&\leqslant \sum_{j=0}^{\infty} r^{j+1} \mathbb{P} \left\{r^{j+1} < \lvert X \rvert^{p} \leqslant r^{j+2} \right\} \\
&= \sum_{j=0}^{\infty} \mathbb{E} \left(r^{j+1} I_{\left\{r^{j+1} < \lvert X \rvert^{p} \leqslant r^{j+2} \right\}} \right) \\
&\leqslant \sum_{j=0}^{\infty} \mathbb{E} \lvert X \rvert^{p} I_{\left\{r^{j+1} < \lvert X \rvert^{p} \leqslant r^{j+2} \right\}} \\
&= \mathbb{E} \lvert X \rvert^{p} I_{\left\{\lvert X \rvert^{p} > r \right\}} \\
&\leqslant \mathbb{E} \lvert X \rvert^{p}
\end{align*}
and (c) of Theorem~\ref{thr:3.1} holds. By \eqref{eq:3.4}, we also obtain
\begin{align*}
&\sum_{m=0}^{\infty} \sum_{n=r^{m}}^{r^{m+1} - 1} \sum_{k=1}^{m+1} \frac{c_{n} r^{m}}{a_{m,k}^{2}} \mathbb{E} X^{2} I_{\left\{\lvert X \rvert \leqslant b_{r^{k}} \right\}} \\
&\quad \leqslant (r - 1) \sum_{m=0}^{\infty} \sum_{k=1}^{m+1} r^{(1 - 2/\alpha) m} r^{2(1/\alpha - 1/p) k} \mathbb{E} X^{2} I_{\left\{\lvert X \rvert \leqslant r^{k/p} \right\}} \\
&\quad = (r - 1) \sum_{k=1}^{\infty} \sum_{m=k-1}^{\infty} r^{(1 - 2/\alpha) m} r^{2(1/\alpha - 1/p) k} \mathbb{E} X^{2} I_{\left\{\lvert X \rvert \leqslant r^{k/p} \right\}} \\
&\quad = \frac{(r - 1)r^{2/\alpha - 1}}{1 - r^{1 - 2/\alpha}} \sum_{k=1}^{\infty} r^{(1 - 2/p)k} \mathbb{E} X^{2} I_{\left\{\lvert X \rvert \leqslant r^{k/p} \right\}} \\
&\quad = \frac{(r - 1)r^{2/\alpha - 1}}{1 - r^{1 - 2/\alpha}} \sum_{k=1}^{\infty} r^{(1 - 2/p)k} \mathbb{E} X^{2} I_{\left\{\lvert X \rvert \leqslant 1 \right\}} \\
&\qquad + \frac{(r - 1)r^{2/\alpha - 1}}{1 - r^{1 - 2/\alpha}} \sum_{k=1}^{\infty} \sum_{j=1}^{k} r^{(1 - 2/p)k} \mathbb{E} X^{2} I_{\left\{r^{(j-1)/p} < \lvert X \rvert \leqslant r^{j/p} \right\}} \\
&\quad \leqslant \frac{(r - 1)r^{2/\alpha - 2/p}}{\left(1 - r^{1 - 2/\alpha} \right)\left(1 - r^{1 - 2/p} \right)} + \frac{(r - 1)r^{2/\alpha - 1}}{1 - r^{1 - 2/\alpha}} \sum_{j=1}^{\infty} \underset{=\frac{r^{j(1 - 2/p)}}{1 - r^{1 - 2/p}}}{\underbrace{\sum_{k=j}^{\infty} r^{k(1 - 2/p)}}} \mathbb{E} X^{2} I_{\left\{r^{(j-1)/p} < \lvert X \rvert \leqslant r^{j/p} \right\}} \\
&\quad \leqslant \frac{(r - 1)r^{2/\alpha - 2/p}}{\left(1 - r^{1 - 2/\alpha} \right)\left(1 - r^{1 - 2/p} \right)} + \frac{(r - 1)r^{2/\alpha - 1}}{\left(1 - r^{1 - 2/\alpha} \right)\left(1 - r^{1 - 2/p} \right)} \sum_{j=1}^{\infty} \mathbb{E} \lvert X \rvert^{p} I_{\left\{r^{(j-1)/p} < \lvert X \rvert \leqslant r^{j/p} \right\}} \\
&\quad = \frac{(r - 1)r^{2/\alpha - 2/p}}{\left(1 - r^{1 - 2/\alpha} \right)\left(1 - r^{1 - 2/p} \right)} + \frac{(r - 1)r^{2/\alpha - 1}}{\left(1 - r^{1 - 2/\alpha} \right)\left(1 - r^{1 - 2/p} \right)} \mathbb{E} \lvert X \rvert^{p} I_{\left\{\lvert X \rvert > 1 \right\}} \\
&\quad \leqslant \frac{(r - 1)r^{2/\alpha - 2/p}}{\left(1 - r^{1 - 2/\alpha} \right)\left(1 - r^{1 - 2/p} \right)} + \frac{(r - 1)r^{2/\alpha - 1}}{\left(1 - r^{1 - 2/\alpha} \right)\left(1 - r^{1 - 2/p} \right)} \mathbb{E} \lvert X \rvert^{p}
\end{align*}
and
\begin{align*}
&\sum_{m=0}^{\infty} \sum_{n=r^{m}}^{r^{m+1} - 1} \sum_{k=1}^{m+1} \frac{c_{n} r^{m} r^{2k/p}}{a_{m,k}^{2}} \mathbb{P} \big\{\lvert X \rvert > b_{r^{(k - 1)/p}} \big\} \\
&\quad \leqslant (r - 1) \sum_{m=0}^{\infty} \sum_{k=1}^{m+1} r^{(2/\alpha - 2/p)k - 2m/\alpha} r^{m} r^{2k/p} \mathbb{P} \big\{\lvert X \rvert > r^{(k - 1)/p} \big\} \\
&\quad = (r - 1) \sum_{k=1}^{\infty} \sum_{m=k - 1}^{\infty} r^{(1 - 2/\alpha)m} r^{2k/\alpha} \mathbb{P} \big\{\lvert X \rvert > r^{(k - 1)/p} \big\} \\
&\quad = \frac{(r - 1)r^{2/\alpha - 1}}{1 - r^{1 - 2/\alpha}} \sum_{k=1}^{\infty} r^{k} \mathbb{P} \big\{\lvert X \rvert > r^{(k - 1)/p} \big\} \\
&\quad = \frac{(r - 1)r^{2/\alpha - 1}}{1 - r^{1 - 2/\alpha}} \sum_{k=1}^{\infty} \sum_{j=k - 1}^{\infty} r^{k} \mathbb{P} \big\{r^{j/p} < \lvert X \rvert \leqslant r^{(j+1)/p} \big\} \\
&\quad = \frac{(r - 1)r^{2/\alpha - 1}}{1 - r^{1 - 2/\alpha}} \sum_{j=0}^{\infty} \sum_{k=1}^{j+1}  r^{k} \mathbb{P} \big\{r^{j/p} < \lvert X \rvert \leqslant r^{(j+1)/p} \big\} \\
&\quad \leqslant \frac{r^{2/\alpha + 1}}{1 - r^{1 - 2/\alpha}} \sum_{j=0}^{\infty} r^{j} \mathbb{P} \big\{r^{j/p} < \lvert X \rvert \leqslant r^{(j+1)/p} \big\} \\
&\quad \leqslant \frac{r^{2/\alpha + 1}}{1 - r^{1 - 2/\alpha}} \sum_{j=0}^{\infty} \mathbb{E} \lvert X \rvert^{p} I_{\left\{r^{j/p} < \lvert X \rvert \leqslant r^{(j+1)/p} \right\}} \\
&\quad = \frac{r^{2/\alpha + 1}}{1 - r^{1 - 2/\alpha}} \mathbb{E} \lvert X \rvert^{p} I_{\left\{\lvert X \rvert > 1 \right\}} \\
&\quad \leqslant \frac{r^{2/\alpha + 1}}{1 - r^{1 - 2/\alpha}} \mathbb{E} \lvert X \rvert^{p}
\end{align*}
so that assumptions (d), (e) of Theorem~\ref{thr:3.1} both hold. Finally, we have
\begin{equation}\label{eq:3.5}
\begin{split}
&\sum_{m=0}^{\infty} \sum_{n=r^{m}}^{r^{m+1} - 1} \sum_{k=1}^{m+1} \frac{c_{n}}{a_{m,k}^{2}} \max_{1 \leqslant \ell \leqslant r-1} \sum_{h=0}^{r^{m-k+1} - 1} \left[\sum_{1 + hr^{k} \leqslant i < j \leqslant h r^{k} + \ell r^{k-1}} G_{X_{i},X_{j}}(b_{r^{k-1}}) \right]^{+} \\
&\quad \leqslant \sum_{m=0}^{\infty} \sum_{n=r^{m}}^{r^{m+1} - 1} \sum_{k=1}^{m+1} \frac{r^{(2/\alpha - 2/p)k - 2m/\alpha}}{n} \max_{1 \leqslant \ell \leqslant r-1} \sum_{h=0}^{r^{m-k+1} - 1} \sum_{1 + hr^{k} \leqslant i < j \leqslant h r^{k} + \ell r^{k-1}} G_{X_{i},X_{j}}^{+} \big(r^{(k - 1)/p} \big) \\
&\quad \leqslant (r - 1) \sum_{m=0}^{\infty} \sum_{k=1}^{m+1} r^{(2/\alpha - 2/p)k - 2m/\alpha} \sum_{1 \leqslant i < j \leqslant r^{m+1}} G_{X_{i},X_{j}}^{+} \big(r^{(k - 1)/p} \big) \\
&\quad = (r - 1)r^{\alpha/2} \sum_{m=1}^{\infty} \sum_{k=1}^{m} r^{(2/\alpha - 2/p)k - 2m/\alpha} \sum_{1 \leqslant i < j \leqslant r^{m}} G_{X_{i},X_{j}}^{+} \big(r^{(k - 1)/p} \big) \\
&\quad = (r - 1)r^{\alpha/2} \sum_{1 \leqslant i < j \leqslant \infty} \sum_{k=1}^{\infty} \sum_{m=k}^{\infty} r^{(2/\alpha - 2/p)k - 2m/\alpha} I_{\left\{j \leqslant r^{m}  \right\}} G_{X_{i},X_{j}}^{+} \big(r^{(k - 1)/p} \big) \\
&\quad \leqslant \frac{(r - 1)r^{\alpha/2}}{1 - r^{-2/\alpha}} \sum_{1 \leqslant i < j \leqslant \infty} \sum_{k=1}^{\infty} r^{(2/\alpha - 2/p)k - 2(k \vee \log_{r} j)/\alpha} G_{X_{i},X_{j}}^{+} \big(r^{(k - 1)/p} \big) < \infty
\end{split}
\end{equation}
as well as
\begin{align}
&\sum_{m=0}^{\infty} \sum_{n=r^{m}}^{r^{m+1} - 1} \sum_{k=1}^{m+1} \frac{c_{n}}{a_{m,k}^{2}} \sum_{h=0}^{r^{m-k+1} - 1} \left[\sum_{1 + hr^{k} \leqslant i < j \leqslant h r^{k} + r^{k}} H_{X_{i},X_{j}}(b_{r^{k-1}},b_{r^{k}}) \right]^{+} \notag \\
&\quad \leqslant \sum_{m=0}^{\infty} \sum_{n=r^{m}}^{r^{m+1} - 1} \sum_{k=1}^{m+1} \frac{r^{(2/\alpha - 2/p)k - 2m/\alpha}}{n} \sum_{h=0}^{r^{m-k+1} - 1} \sum_{1 + hr^{k} \leqslant i < j \leqslant h r^{k} + r^{k}} H_{X_{i},X_{j}}^{+} \big(r^{(k - 1)/p}, r^{k/p} \big) \notag \\
&\quad \leqslant (r - 1) \sum_{m=0}^{\infty} \sum_{k=1}^{m+1} r^{(2/\alpha - 2/p)k - 2m/\alpha} \sum_{1 \leqslant i < j \leqslant r^{m+1}} H_{X_{i},X_{j}}^{+} \big(r^{(k - 1)/p}, r^{k/p} \big) \notag \\
&\quad = (r - 1)r^{\alpha/2} \sum_{m=1}^{\infty} \sum_{k=1}^{m} r^{(2/\alpha - 2/p)k - 2m/\alpha} \sum_{1 \leqslant i < j \leqslant r^{m}} H_{X_{i},X_{j}}^{+} \big(r^{(k - 1)/p}, r^{k/p} \big) \label{eq:3.6} \\
&\quad = (r - 1)r^{\alpha/2} \sum_{1 \leqslant i < j \leqslant \infty} \sum_{k=1}^{\infty} \sum_{m=k}^{\infty} r^{(2/\alpha - 2/p)k - 2m/\alpha} I_{\left\{j \leqslant r^{m}  \right\}} H_{X_{i},X_{j}}^{+} \big(r^{(k - 1)/p}, r^{k/p} \big) \notag \\
&\quad \leqslant \frac{(r - 1)r^{\alpha/2}}{1 - r^{-2/\alpha}} \sum_{1 \leqslant i < j \leqslant \infty} \sum_{k=1}^{\infty} r^{(2/\alpha - 2/p)k - 2(k \vee \log_{r} j)/\alpha} H_{X_{i},X_{j}}^{+} \big(r^{(k - 1)/p}, r^{k/p} \big) < \infty, \notag
\end{align}
entailing assumptions (f) and (g) of Theorem~\ref{thr:3.1}. The proof is complete.
\end{proof}

\subsection{Sufficient conditions for assumptions (f) and (g)}\label{subsec:3.1}

\indent

We now present sufficient (moment) conditions which guarantee assumptions (f) and (g) of Theorem~\ref{thr:3.1}. Suppose a sequence of random variables $\{X_{n}, \, n \geqslant 1 \}$ for which there is a (fixed) constant $C > 0$ such that
\begin{equation}\label{eq:3.7}
\abs{\sum_{m \leqslant i < j \leqslant n} G_{X_{i},X_{j}}(b_{r^{k-1}})} \leqslant C \sum_{j=m}^{n} \mathbb{E} \lvert g_{b_{r^{k-1}}}(X_{j}) \rvert^{2}
\end{equation}
for all $1 \leqslant m \leqslant n$ (recall that $G_{X_{i},X_{j}}(b_{r^{k-1}}) := \mathrm{Cov} (g_{b_{r^{k-1}}}(X_{i}), g_{b_{r^{k-1}}}(X_{j}))$). Thus,
\begin{align*}
&\max_{1 \leqslant \ell \leqslant r-1} \sum_{h=0}^{r^{m-k+1} - 1} \left[\sum_{1 + hr^{k} \leqslant i < j \leqslant h r^{k} + \ell r^{k-1}} G_{X_{i},X_{j}}(b_{r^{k-1}}) \right]^{+} \\
&\quad \leqslant \max_{1 \leqslant \ell \leqslant r-1} \sum_{h=0}^{r^{m-k+1} - 1} \abs{\sum_{1 + hr^{k} \leqslant i < j \leqslant h r^{k} + \ell r^{k-1}} G_{X_{i},X_{j}}(b_{r^{k-1}})} \\
&\quad \leqslant \max_{1 \leqslant \ell \leqslant r-1} C \sum_{h=0}^{r^{m-k+1} - 1} \sum_{j=1 + hr^{k}}^{h r^{k} + \ell r^{k-1}} \mathbb{E} \lvert g_{b_{r^{k-1}}}(X_{j}) \rvert^{2} \\
&\quad \leqslant C \sum_{j=1}^{r^{m+1}} \mathbb{E} \lvert g_{b_{r^{k-1}}}(X_{j}) \rvert^{2} \\
&\quad = C \sum_{j=1}^{r^{m+1}} \left[\mathbb{E} X_{j}^{2} I_{\left\{\lvert X_{j} \rvert \leqslant b_{r^{k-1}} \right\}} + b_{r^{k-1}}^{2} \mathbb{P} \big\{\lvert X_{j} \rvert > b_{r^{k-1}} \big\} \right]
\end{align*}
and if we additionally assume \eqref{eq:3.7} in Theorem~\ref{thr:3.1}, then condition (f) can be removed.

On the other hand, by putting $f_{L,K}(t) := g_{K}(t) - g_{L}(t)$ and admitting the existence of a constant $C > 0$ (which is not necessarily the same on each appearance) such that the following inequalities hold
\begin{align}
\abs{\sum_{m \leqslant i < j \leqslant n} \mathrm{Cov}\big(f_{b_{r^{k-1}},b_{r^{k}}}(X_{i}), f_{b_{r^{k-1}},b_{r^{k}}}(X_{j}) \big)} \leqslant C \sum_{j=m}^{n} \mathbb{E} \lvert f_{b_{r^{k-1}},b_{r^{k}}}(X_{j}) \rvert^{2}, \notag \\
\abs{\sum_{m \leqslant i < j \leqslant n} \mathrm{Cov}\big(f_{b_{r^{k-1}},b_{r^{k}}}^{+}(X_{i}), f_{b_{r^{k-1}},b_{r^{k}}}^{+}(X_{j}) \big)} \leqslant C \sum_{j=m}^{n} \mathbb{E} \big[f_{b_{r^{k-1}},b_{r^{k}}}^{+}(X_{j}) \big]^{2}, \label{eq:3.8} \\
\abs{\sum_{m \leqslant i < j \leqslant n} \mathrm{Cov}\big(f_{b_{r^{k-1}},b_{r^{k}}}^{-}(X_{i}), f_{b_{r^{k-1}},b_{r^{k}}}^{-}(X_{j}) \big)} \leqslant C \sum_{j=m}^{n} \mathbb{E} \big[ f_{b_{r^{k-1}},b_{r^{k}}}^{-}(X_{j}) \big]^{2} \notag
\end{align}
for all $1 \leqslant m \leqslant n$, we get
\begin{equation}\label{eq:3.9}
\begin{split}
&\Bigg\lvert \sum_{1 + h r^{k} \leqslant i < j \leqslant h r^{k} + r^{k}} H_{X_{i},X_{j}}\big(b_{r^{k-1}},b_{r^{k}} \big) \Bigg\rvert \\
&\quad = \Bigg\lvert \sum_{1 + h r^{k} \leqslant i < j \leqslant h r^{k} + r^{k}} \Big[\mathrm{Cov}\big(f_{b_{r^{k-1}},b_{r^{k}}}^{+}(X_{i}), f_{b_{r^{k-1}},b_{r^{k}}}^{+}(X_{j}) \big) + \mathrm{Cov}\big(f_{b_{r^{k-1}},b_{r^{k}}}^{+}(X_{i}), f_{b_{r^{k-1}},b_{r^{k}}}^{-}(X_{j}) \big) \bigg. \Bigg. \\
&\qquad \Bigg.\bigg. + \mathrm{Cov}\big(f_{b_{r^{k-1}},b_{r^{k}}}^{-}(X_{i}), f_{b_{r^{k-1}},b_{r^{k}}}^{+}(X_{j}) \big) + \mathrm{Cov}\big(f_{b_{r^{k-1}},b_{r^{k}}}^{-}(X_{i}), f_{b_{r^{k-1}},b_{r^{k}}}^{-}(X_{j}) \big) \Big] \Bigg\rvert \\
&\quad \leqslant \Bigg\lvert \sum_{1 + h r^{k} \leqslant i < j \leqslant h r^{k} + r^{k}} \mathrm{Cov}\big(f_{b_{r^{k-1}},b_{r^{k}}}^{+}(X_{i}), f_{b_{r^{k-1}},b_{r^{k}}}^{+}(X_{j}) \big) \Bigg\rvert \\
&\qquad + \Bigg\lvert\sum_{1 + h r^{k} \leqslant i < j \leqslant h r^{k} + r^{k}} \mathrm{Cov}\big(f_{b_{r^{k-1}},b_{r^{k}}}^{-}(X_{i}), f_{b_{r^{k-1}},b_{r^{k}}}^{-}(X_{j}) \big) \Bigg\rvert \\
&\qquad + \Bigg\lvert \sum_{1 + h r^{k} \leqslant i < j \leqslant h r^{k} + r^{k}} \left[ \mathrm{Cov}\big(f_{b_{r^{k-1}},b_{r^{k}}}^{+}(X_{i}), f_{b_{r^{k-1}},b_{r^{k}}}^{-}(X_{j}) \big) + \mathrm{Cov}\big(f_{b_{r^{k-1}},b_{r^{k}}}^{-}(X_{i}), f_{b_{r^{k-1}},b_{r^{k}}}^{+}(X_{j}) \big) \right] \Bigg\rvert \\
&\quad \leqslant C \sum_{j=1 + h r^{k}}^{h r^{k} + r^{k}} \mathbb{E} \big[ f_{b_{r^{k-1}},b_{r^{k}}}^{+}(X_{j}) \big]^{2} + C \sum_{j=1 + h r^{k}}^{h r^{k} + r^{k}} \mathbb{E} \big[ f_{b_{r^{k-1}},b_{r^{k}}}^{-}(X_{j}) \big]^{2} + 3C \sum_{j=1 + h r^{k}}^{h r^{k} + r^{k}} \mathbb{E} \lvert f_{b_{r^{k-1}},b_{r^{k}}}(X_{j}) \rvert^{2} \\
&\quad \leqslant 5C \sum_{j=1 + h r^{k}}^{h r^{k} + r^{k}} \mathbb{E} \lvert f_{b_{r^{k-1}},b_{r^{k}}}(X_{j}) \rvert^{2} \\
&\quad \leqslant 5C \sum_{j=1 + h r^{k}}^{h r^{k} + r^{k}} \left[\mathbb{E} X_{j}^{2} I_{\left\{b_{r^{k-1}} < \lvert X_{j} \rvert \leqslant b_{r^{k}} \right\}} + b_{r^{k}}^{2} \mathbb{P} \big\{\lvert X_{j} \rvert > b_{r^{k}} \big\} \right]
\end{split}
\end{equation}
because
\begin{align*}
&C \sum_{j=m}^{n} \mathbb{E} \lvert f_{b_{r^{k-1}},b_{r^{k}}}(X_{j}) \rvert^{2} \\
&\quad \geqslant \Bigg\lvert\sum_{1 + h r^{k} \leqslant i < j \leqslant h r^{k} + r^{k}} \mathrm{Cov}\big(f_{b_{r^{k-1}},b_{r^{k}}}(X_{i}), f_{b_{r^{k-1}},b_{r^{k}}}(X_{j}) \big)\Bigg\rvert \\
&\quad = \Bigg\lvert \sum_{1 + h r^{k} \leqslant i < j \leqslant h r^{k} + r^{k}} \Big[\mathrm{Cov}\big(f_{b_{r^{k-1}},b_{r^{k}}}^{+}(X_{i}), f_{b_{r^{k-1}},b_{r^{k}}}^{+}(X_{j}) \big) - \mathrm{Cov}\big(f_{b_{r^{k-1}},b_{r^{k}}}^{+}(X_{i}), f_{b_{r^{k-1}},b_{r^{k}}}^{-}(X_{j}) \big) \bigg. \Bigg. \\
&\qquad \Bigg.\bigg. - \mathrm{Cov}\big(f_{b_{r^{k-1}},b_{r^{k}}}^{-}(X_{i}), f_{b_{r^{k-1}},b_{r^{k}}}^{+}(X_{j}) \big) + \mathrm{Cov}\big(f_{b_{r^{k-1}},b_{r^{k}}}^{-}(X_{i}), f_{b_{r^{k-1}},b_{r^{k}}}^{-}(X_{j}) \big) \Big] \Bigg\rvert \\
&\quad \geqslant \Bigg\lvert \sum_{1 + h r^{k} \leqslant i < j \leqslant h r^{k} + r^{k}} \Big[\mathrm{Cov}\big(f_{b_{r^{k-1}},b_{r^{k}}}^{+}(X_{i}), f_{b_{r^{k-1}},b_{r^{k}}}^{-}(X_{j}) \big) + \mathrm{Cov}\big(f_{b_{r^{k-1}},b_{r^{k}}}^{-}(X_{i}), f_{b_{r^{k-1}},b_{r^{k}}}^{+}(X_{j}) \big) \Big]\Bigg\rvert \\
&\qquad - \Bigg\lvert \sum_{1 + h r^{k} \leqslant i < j \leqslant h r^{k} + r^{k}} \Big[\mathrm{Cov}\big(f_{b_{r^{k-1}},b_{r^{k}}}^{+}(X_{i}), f_{b_{r^{k-1}},b_{r^{k}}}^{+}(X_{j}) \big) + \mathrm{Cov}\big(f_{b_{r^{k-1}},b_{r^{k}}}^{-}(X_{i}), f_{b_{r^{k-1}},b_{r^{k}}}^{-}(X_{j}) \big) \Big] \Bigg\rvert \\
&\quad \geqslant \Bigg\lvert \sum_{1 + h r^{k} \leqslant i < j \leqslant h r^{k} + r^{k}} \Big[\mathrm{Cov}\big(f_{b_{r^{k-1}},b_{r^{k}}}^{+}(X_{i}), f_{b_{r^{k-1}},b_{r^{k}}}^{-}(X_{j}) \big) + \mathrm{Cov}\big(f_{b_{r^{k-1}},b_{r^{k}}}^{-}(X_{i}), f_{b_{r^{k-1}},b_{r^{k}}}^{+}(X_{j}) \big) \Big]\Bigg\rvert \\
&\qquad - C \sum_{j=m}^{n} \mathbb{E} \big[ f_{b_{r^{k-1}},b_{r^{k}}}^{+}(X_{j}) \big]^{2} - C \sum_{j=m}^{n} \mathbb{E} \big[ f_{b_{r^{k-1}},b_{r^{k}}}^{-}(X_{j}) \big]^{2} \\
&\quad \geqslant \Bigg\lvert \sum_{1 + h r^{k} \leqslant i < j \leqslant h r^{k} + r^{k}} \Big[\mathrm{Cov}\big(f_{b_{r^{k-1}},b_{r^{k}}}^{+}(X_{i}), f_{b_{r^{k-1}},b_{r^{k}}}^{-}(X_{j}) \big) + \mathrm{Cov}\big(f_{b_{r^{k-1}},b_{r^{k}}}^{-}(X_{i}), f_{b_{r^{k-1}},b_{r^{k}}}^{+}(X_{j}) \big) \Big]\Bigg\rvert \\
&\qquad - 2C \sum_{j=m}^{n} \mathbb{E} \lvert f_{b_{r^{k-1}},b_{r^{k}}}(X_{j}) \rvert^{2}
\end{align*}
and
\begin{align*}
\sum_{j=1 + h r^{k}}^{h r^{k} + r^{k}} \mathbb{E} \lvert f_{b_{r^{k-1}},b_{r^{k}}}(X_{j}) \rvert^{2} & = \sum_{j=1 + h r^{k}}^{h r^{k} + r^{k}} \mathbb{E} \big[ h_{b_{r^{k-1}},b_{r^{k}}}(X_{j}) \big]^{2} \\
&\leqslant \sum_{j=1 + h r^{k}}^{h r^{k} + r^{k}} \left[\mathbb{E} X_{j}^{2} I_{\left\{b_{r^{k-1}} < \lvert X_{j} \rvert \leqslant b_{r^{k}} \right\}} + b_{r^{k}}^{2} \mathbb{P} \big\{\lvert X_{j} \rvert > b_{r^{k}} \big\} \right].
\end{align*}
Therefore, by \eqref{eq:3.9} we obtain
\begin{align*}
&\sum_{h=0}^{r^{m-k+1} - 1} \left[\sum_{1 + hr^{k} \leqslant i < j \leqslant h r^{k} + r^{k}} H_{X_{i},X_{j}}(b_{r^{k-1}},b_{r^{k}}) \right]^{+} \\
&\quad \leqslant \sum_{h=0}^{r^{m-k+1} - 1} \abs{\sum_{1 + hr^{k} \leqslant i < j \leqslant h r^{k} + r^{k}} H_{X_{i},X_{j}}(b_{r^{k-1}},b_{r^{k}})}  \\
&\quad \leqslant C \sum_{h=0}^{r^{m-k+1} - 1} \sum_{j=1 + h r^{k}}^{h r^{k} + r^{k}} \left[\mathbb{E} X_{j}^{2} I_{\left\{b_{r^{k-1}} < \lvert X_{j} \rvert \leqslant b_{r^{k}} \right\}} + b_{r^{k}}^{2} \mathbb{P} \big\{\lvert X_{j} \rvert > b_{r^{k}} \big\} \right] \\
&\quad = C \sum_{j=1}^{r^{m+1}} \left[\mathbb{E} X_{j}^{2} I_{\left\{b_{r^{k-1}} < \lvert X_{j} \rvert \leqslant b_{r^{k}} \right\}} + b_{r^{k}}^{2} \mathbb{P} \big\{\lvert X_{j} \rvert > b_{r^{k}} \big\} \right].
\end{align*}
Hence, by assuming \eqref{eq:3.8} in Theorem~\ref{thr:3.1} one can withdraw condition (g).

Furthermore, it is worthy to note that both \eqref{eq:3.7} and \eqref{eq:3.8} are assured by a familiar moment inequality.  Let $\varphi$ be an arbitrary nondecreasing real-valued function defined on whole real line. As a matter of fact, all sequences of random variables $\{X_{n}, \, n \geqslant 1 \}$ satisfying $\mathbb{E} [\varphi(X_{n})]^{2} < \infty$ for each $n \geqslant 1$, and for which there exists a constant $C > 0$ such that
\begin{equation}\label{eq:3.10}
\mathbb{E} \Bigg\{\sum_{j=m}^{n} \left[\varphi(X_{j}) - \mathbb{E} \varphi(X_{j}) \right]\Bigg\}^{2} \leqslant C \sum_{j=m}^{n} \mathbb{E} \left[\varphi(X_{j}) - \mathbb{E} \varphi(X_{j}) \right]^{2}
\end{equation}
for all $1 \leqslant m \leqslant n$, fulfill \eqref{eq:3.7} and \eqref{eq:3.8}. Indeed, supposing \eqref{eq:3.10} it follows
\begin{align*}
&2\Bigg\lvert\sum_{m \leqslant i < j \leqslant n} \mathrm{Cov}(\varphi(X_{i}),\varphi(X_{j})) \Bigg\rvert - \sum_{j=m}^{n} \mathbb{E} \left[\varphi(X_{j}) \right]^{2} \\
&\quad \leqslant 2\Bigg\lvert\sum_{m \leqslant i < j \leqslant n} \mathrm{Cov}(\varphi(X_{i}),\varphi(X_{j})) \Bigg\rvert - \sum_{j=m}^{n} \mathbb{V} \left[\varphi(X_{j}) \right] \\
&\quad = 2\Bigg\lvert\sum_{m \leqslant i < j \leqslant n} \mathbb{E}[\varphi(X_{i}) - \mathbb{E} \varphi(X_{i})][\varphi(X_{j}) - \mathbb{E} \varphi(X_{j})] \Bigg\rvert - \sum_{j=m}^{n} \mathbb{E} \left[\varphi(X_{j}) - \mathbb{E} \varphi(X_{j}) \right]^{2} \\
&\quad \leqslant \Bigg\lvert\sum_{j=m}^{n} \mathbb{E} \left[\varphi(X_{j}) - \mathbb{E} \varphi(X_{j}) \right]^{2} + 2 \sum_{m \leqslant i < j \leqslant n} \mathbb{E}[\varphi(X_{i}) - \mathbb{E} \varphi(X_{i})][\varphi(X_{j}) - \mathbb{E} \varphi(X_{j})] \Bigg\rvert \\
&\quad = \mathbb{E} \Bigg\{\sum_{j=m}^{n} \left[\varphi(X_{j}) - \mathbb{E} \varphi(X_{j}) \right] \Bigg\}^{2} \\
&\quad \leqslant C \sum_{j=m}^{n} \mathbb{E} \left[\varphi(X_{j}) - \mathbb{E} \varphi(X_{j}) \right]^{2} \\
&\quad \leqslant C \sum_{j=m}^{n} \mathbb{E} \left[\varphi(X_{j}) \right]^{2}.
\end{align*}

All in all, excluding \eqref{eq:3.5} and \eqref{eq:3.6}, every steps in proof of Corollary~\ref{cor:3.1} remain valid for sequences of random variables $\{X_{n}, \, n \geqslant 1 \}$ stochastically dominated by a random variable $X \in \mathscr{L}_{p}$, $1 \leqslant p < 2$ satisfying \eqref{eq:3.7} and \eqref{eq:3.8} with $b_{n} = n^{1/p}$. Thereby, we have the ensuing:

\begin{corollary}\label{cor:3.2}
Let $1 \leqslant p < 2$, and $\{X_{n}, \, n \geqslant 1 \}$ be a sequence of random variables stochastically dominated by a random variable $X \in \mathscr{L}_{p}$. If there is a constant $C > 0$ such that
\begin{equation}\label{eq:3.11}
\Bigg\lvert\sum_{m \leqslant i < j \leqslant n} \mathrm{Cov}\big(\varphi(X_{i}),\varphi(X_{j}) \big) \Bigg\rvert \leqslant C \sum_{j=m}^{n} \mathbb{E} [\varphi(X_{j})]^{2}
\end{equation}
for all $1 \leqslant m \leqslant n$ and any nondecreasing real-valued function $ \varphi$ defined on whole real line, then $\sum_{k = 1}^{n} (X_{k} - \mathbb{E} \, X_{k})/n^{1/p} \overset{\textnormal{a.s.}}{\longrightarrow} 0$.
\end{corollary}

\begin{remark}
Of course, Corollary~\ref{cor:3.2} is still true under condition \eqref{eq:3.11} only for $\varphi(t) = g_{r^{(k - 1)/p}}(t)$, $\varphi(t) = f_{r^{(k - 1)/p},r^{k/p}}(t)$, $\varphi(t) = f_{r^{(k - 1)/p},r^{k/p}}^{+}(t)$, and $\varphi(t) = -f_{r^{(k - 1)/p},r^{k/p}}^{+}(t)$ with some (fixed) integer $r > 1$. 
\end{remark}

\subsection{SLLN for quadrant dependent random variables}

\indent

The notion of quadrant dependence was introduced by Lehmann in \cite{Lehmann66}. A sequence $\{X_{n}, \, n \geqslant 1 \}$ of random variables is said to be \emph{pairwise positively quadrant dependent} (pairwise PQD) if
\begin{equation*}
\mathbb{P} \left\{X_{k} \leqslant x_{k}, X_{j} \leqslant x_{j}  \right\} - \mathbb{P} \left\{X_{k} \leqslant x_{k} \right\} \mathbb{P} \left\{X_{j} \leqslant x_{j}  \right\} \geqslant 0
\end{equation*}
for all reals $x_{k}, x_{j}$ and all positive integers $k,j$ such that $k \neq j$. A sequence $\{X_{n}, \, n \geqslant 1 \}$ of random variables is said to be \emph{pairwise negatively quadrant dependent} (pairwise NQD) if
\begin{equation*}
\mathbb{P} \left\{X_{k} \leqslant x_{k}, X_{j} \leqslant x_{j}  \right\} - \mathbb{P} \left\{X_{k} \leqslant x_{k} \right\} \mathbb{P} \left\{X_{j} \leqslant x_{j}  \right\} \leqslant 0
\end{equation*}
for all reals $x_{k}, x_{j}$ and all positive integers $k,j$ such that $k \neq j$. The following is a SLLN for sequences of pairwise PQD random variables.

A remarkable example of the above statistical dependence concept can be found in reliability theory, where most bivariate distributions are positively quadrant dependent and the disregarding of this assumption lead to underestimation/overestimation of the system reliability \cite{Hutchinson90}.

The statement below is a strong limit theorem for sequences of pairwise PQD random variables under both sharpen norming constants and moment condition.

\begin{corollary}\label{cor:3.3}
Let $1 < p < 2$ and $\{X_{n}, \, n \geqslant 1 \}$ be a sequence of pairwise PQD random variables stochastically dominated by a random variable $X \in \mathscr{L}_{p}$. If
\begin{equation*}
\sum_{1 \leqslant i < j \leqslant \infty} \sum_{n=j}^{\infty} n^{-1 - 2/\alpha} G_{X_{i},X_{j}}\big(n^{1/p} \big) < \infty
\end{equation*}
for some $p < \alpha < 2$, then $\sum_{k = 1}^{n} (X_{k} - \mathbb{E} \, X_{k})/n^{1/p} \overset{\textnormal{a.s.}}{\longrightarrow} 0$.
\end{corollary}

\begin{proof}
Since $\Delta_{X_{i},X_{j}}(u,v) \geqslant 0$ for every $i \neq j$ and all reals $u,v$, we have
\begin{equation*}
G_{X_{i},X_{j}} \big(r^{k/p} \big) \geqslant G_{X_{i},X_{j}} \big(r^{(k - 1)/p} \big) \geqslant 0
\end{equation*}
which leads to
\begin{equation*}
G_{X_{i},X_{j}}^{+} \big(r^{(k - 1)/p} \big) + H_{X_{i},X_{j}}^{+} \big(r^{(k - 1)/p}, r^{k/p} \big) \leqslant G_{X_{i},X_{j}} \big(r^{k/p} \big).
\end{equation*}
From Lemma 4 of \cite{Louhichi00}, it follows
\begin{align}
&\sum_{1 \leqslant i < j \leqslant \infty} \sum_{k=1}^{\infty} r^{(2/\alpha - 2/p)k - 2(k \vee \log_{r} j)/\alpha} \left[G_{X_{i},X_{j}}^{+} \big(r^{(k - 1)/p} \big) + H_{X_{i},X_{j}}^{+} \big(r^{(k - 1)/p}, r^{k/p} \big) \right] \notag \\
&\quad \leqslant \sum_{1 \leqslant i < j \leqslant \infty} \sum_{k=1}^{\infty} r^{(2/\alpha - 2/p)k - 2(k \vee \log_{r} j)/\alpha} G_{X_{i},X_{j}} \big(r^{k/p} \big) \notag \\
&\quad = \sum_{1 \leqslant i < j \leqslant \infty} \sum_{k=1}^{\infty} r^{(2/\alpha - 2/p)k - 2(k \vee \log_{r} j)/\alpha} I_{\left\{r^{k} < j \right\}} G_{X_{i},X_{j}} \big(r^{k/p} \big) \notag \\
&\qquad + \sum_{1 \leqslant i < j \leqslant \infty} \sum_{k=1}^{\infty} r^{(2/\alpha - 2/p)k - 2(k \vee \log_{r} j)/\alpha} I_{\left\{r^{k} \geqslant j \right\}} G_{X_{i},X_{j}} \big(r^{k/p} \big) \notag \\
&\quad = \sum_{1 \leqslant i < j \leqslant \infty} j^{-2/\alpha} \sum_{k=1}^{\infty} r^{(2/\alpha - 2/p)k}  I_{\left\{r^{k} < j \right\}} G_{X_{i},X_{j}} \big(r^{k/p} \big) \label{eq:3.12} \\
&\qquad + \sum_{1 \leqslant i < j \leqslant \infty} \int_{-\infty}^{\infty} \int_{-\infty}^{\infty} \sum_{k=1}^{\infty} r^{-2k/p} I_{\left\{k \geqslant \log_{r}(\lvert u \rvert^{p} \vee \lvert v \rvert^{p} \vee j) \right\}} \Delta_{X_{i},X_{j}}(u,v) \, \mathrm{d}u \mathrm{d}v \notag \\
&\quad \leqslant \sum_{1 \leqslant i < j \leqslant \infty} j^{-2/\alpha} \sum_{k=1}^{\infty} r^{(2/\alpha - 2/p)k} G_{X_{i},X_{j}} \big(j^{1/p} \big) \notag \\
&\qquad + \frac{1}{1 - r^{-2/p}} \sum_{1 \leqslant i < j \leqslant \infty} \int_{-\infty}^{\infty} \int_{-\infty}^{\infty} (\lvert u \rvert^{p} \vee \lvert v \rvert^{p} \vee j)^{-2/p} \Delta_{X_{i},X_{j}}(u,v) \, \mathrm{d}u \mathrm{d}v \notag \\
&\quad = \frac{r^{2/\alpha - 2/p}}{1 - r^{2/\alpha - 2/p}} \sum_{1 \leqslant i < j \leqslant \infty} j^{-2/\alpha} G_{X_{i},X_{j}} \big(j^{1/p} \big) + \frac{2}{1 - r^{-2/p}} \sum_{1 \leqslant i < j \leqslant \infty} \int_{j^{1/p}}^{\infty} t^{-3} G_{X_{i},X_{j}}(t) \, \mathrm{d}t. \notag
\end{align}
Since there are positive constants $C_{1},C_{2}$ such that
\begin{equation*}
C_{1} n^{-1-2/p}G_{X_{i},X_{j}}\big(n^{1/p} \big) \leqslant \int_{n^{1/p}}^{(n+1)^{1/p}} t^{-3} G_{X_{i},X_{j}}(t) \, \mathrm{d}t \leqslant C_{2} (n + 1)^{-1-2/p}G_{X_{i},X_{j}}\big((n + 1)^{1/p} \big)
\end{equation*}
for each $n \geqslant 1$, we have
\begin{equation}\label{eq:3.13}
\begin{split}
\sum_{1 \leqslant i < j \leqslant \infty} \int_{j^{1/p}}^{\infty} t^{-3} G_{X_{i},X_{j}}(t) \, \mathrm{d}t &\leqslant C_{2} \sum_{1 \leqslant i < j \leqslant \infty} \sum_{n=j}^{\infty} n^{-1 - 2/p} G_{X_{i},X_{j}}\big(n^{1/p} \big) \\
&\leqslant C_{2} \sum_{1 \leqslant i < j \leqslant \infty} \sum_{n=j}^{\infty} n^{-1 - 2/\alpha} G_{X_{i},X_{j}}\big(n^{1/p} \big) < \infty.
\end{split}
\end{equation}
Moreover,
\begin{equation}\label{eq:3.14}
\sum_{1 \leqslant i < j \leqslant \infty} j^{-2/\alpha} G_{X_{i},X_{j}}\big(j^{1/p} \big) \leqslant C(\alpha) \sum_{1 \leqslant i < j \leqslant \infty} \sum_{n=j}^{\infty} n^{-1 - 2/\alpha} G_{X_{i},X_{j}}\big(n^{1/p} \big) < \infty.
\end{equation}
Therefore, \eqref{eq:3.12}, \eqref{eq:3.13}, \eqref{eq:3.14} imply condition \eqref{eq:3.2} and the thesis is established by employing Corollary~\ref{cor:3.1}.
\end{proof}

Furthermore, in light of previous subsection~\ref{subsec:3.1}, among the families of random variables for which \eqref{eq:3.10} holds, one can find well-known dependence structures, namely, sequences of pairwise NQD random variables or sequences of extended negatively dependent random variables (for definition and basic properties, see \cite{Lehmann66} and \cite{Liu09}, respectively); with regard to the corresponding proof, we refer \cite{Shen11} and \cite{Wu06}.

\section*{Acknowledgements}

This work is funded by national funds through the FCT - Fundação para a Ciência e a Tecnologia, I.P., under the scope of projects UIDB/04035/2020 (GeoBioTec)

\end{document}